\documentclass[letterpaper, 10 pt, conference]{ieeeconf}  

\IEEEoverridecommandlockouts                              

\usepackage{lipsum}
\usepackage[noadjust]{cite}  
\usepackage{amsmath,amssymb}
\usepackage{aligned-overset}
\usepackage{enumerate}
\usepackage{bm}
\usepackage{amsfonts}
\usepackage{anyfontsize}
\usepackage{cancel}

\usepackage{float}

\newtheorem{theorem}{Theorem}[section]
\newtheorem{lemma}{Lemma}[section]
\newtheorem{proposition}{Proposition}[section]

\newtheorem{definition}{Definition}[section]
\newtheorem{remark}{Remark}[section]
\newtheorem{assumption}{Assumption}[section]
\newtheorem{problem}{Problem}[section]

\newcommand\numberthis{\addtocounter{equation}{1}\tag{\theequation}}
\makeatletter

\usepackage{algpseudocode}
\usepackage{algorithm}

\usepackage{setspace}
    
\usepackage{bm} 

\allowdisplaybreaks  

\usepackage[unicode]{hyperref}
\hypersetup{
    colorlinks,
    citecolor=black,
    filecolor=black,
    linkcolor=black,
    urlcolor=black
}

\usepackage{nccmath} 

\usepackage{graphicx}

\usepackage[font=sf]{caption}
\usepackage{subcaption}
\usepackage{mathrsfs}   

\DeclareMathOperator{\sig}{sig} 
\DeclareMathOperator{\sign}{sign}

\usepackage{accents}

\usepackage{nccmath} 

\usepackage{pgfplots}
\pgfplotsset{compat=1.15}
\usetikzlibrary{patterns}
\usetikzlibrary{bending}
\usetikzlibrary{calc}
\usetikzlibrary{arrows,backgrounds,positioning,fit,petri}
\usetikzlibrary{decorations.pathreplacing,decorations.pathmorphing,decorations.markings,decorations.shapes}
\usetikzlibrary{arrows.meta}
\usetikzlibrary{quotes,angles}
\usetikzlibrary{chains}
\usetikzlibrary{hobby}
\usepackage{tkz-euclide}
\definecolor{ttttff}{rgb}{0.2,0.2,1}
\definecolor{ududff}{rgb}{0.30196078431372547,0.30196078431372547,1}
\definecolor{xfqqff}{rgb}{0.4980392156862745,0,1}
\definecolor{qqzzff}{rgb}{0,0.6,1}
\definecolor{ffxfqq}{rgb}{1,0.4980392156862745,0}
\definecolor{qqqqff}{rgb}{0,0,1}
\definecolor{ffqqff}{rgb}{1,0,1}
\definecolor{qqzzqq}{rgb}{0,0.6,0}
\definecolor{ffqqqq}{rgb}{1,0,0}
\definecolor{qqwuqq}{rgb}{0,0.39215686274509803,0}
\definecolor{zzttqq}{rgb}{0.6,0.2,0}
\definecolor{wqwqwq}{rgb}{0.3764705882352941,0.3764705882352941,0.3764705882352941}
\definecolor{qqttcc}{rgb}{0,0.2,0.8}
\definecolor{windBlue}{rgb}{0.3176470588,0.9254901961,0.9450980392}

\newcommand\norm[1]{\| #1 \|}

\makeatletter
\newcommand{\pushright}[1]{\ifmeasuring@#1\else\omit\hfill$\displaystyle#1$\fi\ignorespaces}
\newcommand{\pushleft}[1]{\ifmeasuring@#1\else\omit$\displaystyle#1$\hfill\fi\ignorespaces}
\makeatother

\newcommand\scalemath[2]{\scalebox{#1}{\mbox{\ensuremath{\displaystyle #2}}}}

\title{\LARGE \bf Predefined-Time Target Localization and Circumnavigation using Bearing-Only Measurements: Theory and Experiments}

\author{Donglin Sui$^{1}$ and Mohammad Deghat$^{1}$
\thanks{$^{1}$The authors are with the School of Mechanical and Manufacturing Engineering, University of New South Wales, Sydney NSW 2052, Australia (e-mail: {\tt\small d.sui@unsw.edu.au, m.deghat@unsw.edu.au}).}%
}

\begin{document}

\maketitle
\thispagestyle{empty}
\pagestyle{empty}

\begin{abstract}
    This paper investigates the problem of controlling an autonomous agent to simultaneously localize and circumnavigate an unknown stationary target using bearing-only measurements (without explicit differentiation). To improve the convergence rate of target estimation, we introduce a novel adaptive target estimator that enables the agent to accurately localize the position of the unknown target with a tunable, predefined convergence time. Following this, we design a controller integrated with the estimator to steer the agent onto a circular trajectory centered at the target with a desired radius. The predefined-time stability of the overall system including the estimation and control errors are rigorously analyzed. Extensive simulations and experiments using unmanned aerial vehicles (UAVs) illustrate the performance and efficacy of the proposed estimation and control algorithms. 
\end{abstract}


\section{INTRODUCTION}
\label{sec:Introduction}

In recent decades, target circumnavigation has attracted significant interest due to its civil and military applications. Typically, target circumnavigation involves guiding a single or a group of agents to orbit a target at a specified radius. While previous research often assumes that agent(s) has access to the target’s state information, this assumption does not hold in many real-world scenarios where the target exhibits uncooperative behaviors and does not share its state. Consequently, the task requires both target localization and navigation control, thereby presenting a dual control problem \cite{feldbaum1963Dual, rantzer2023Dual}.

Bearing measurements are advantageous for localizing unknown targets in stealth operations. Unlike distance measurements, which require detectable signal transmissions toward the target and thereby compromising stealth, bearing measurements are usually passive, relying solely on receiving naturally emitted signals from the target, and can be obtained with low-cost sensors such as monocular cameras. These benefits have made the \textit{Bearing-only Target Localization and Circumnavigation} (\textit{BoTLC}) problem a growing area of research.

\begin{figure}[ht]
    \centering
    \includegraphics[width=1.0\linewidth]{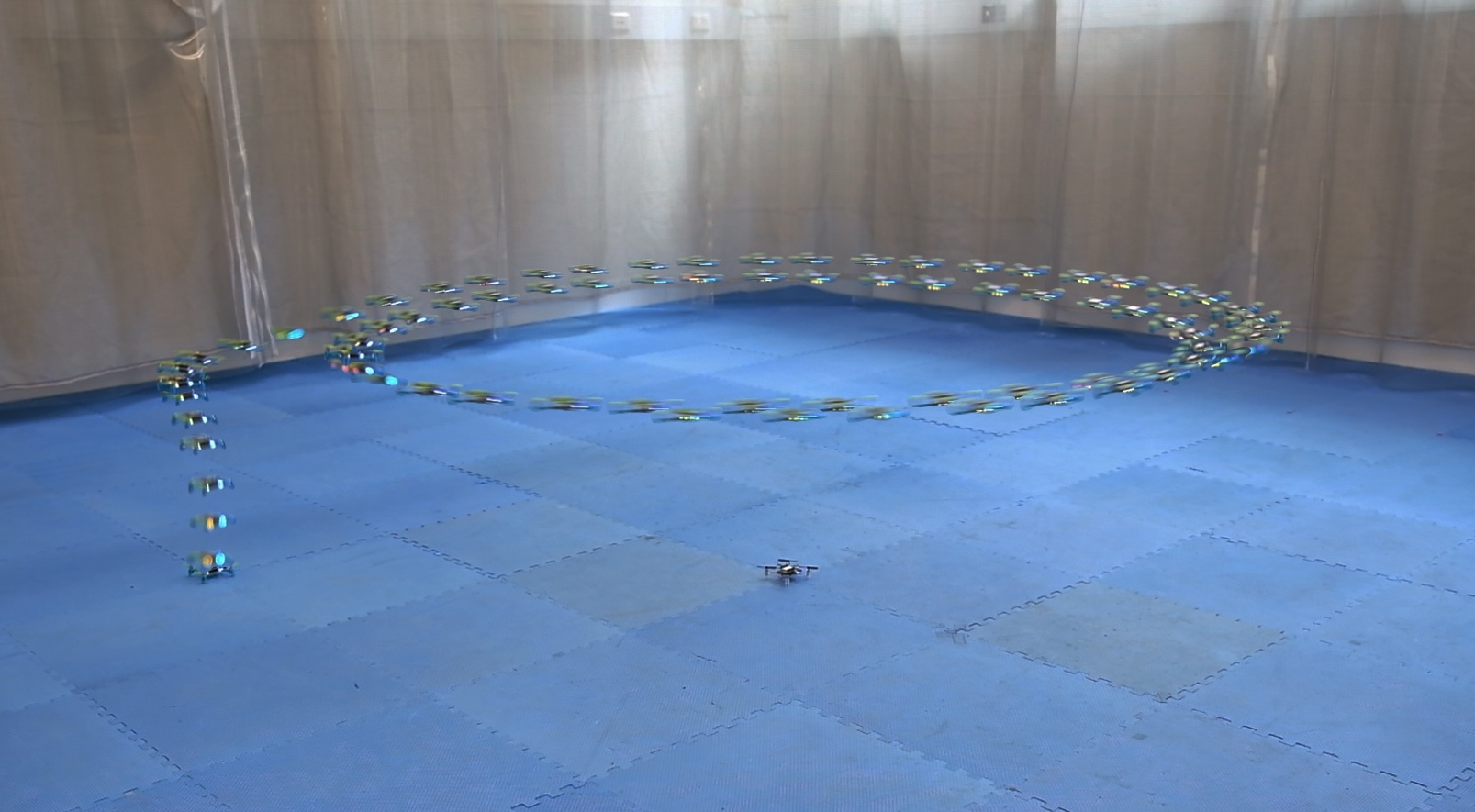} 
    \caption{\label{fig:echo_photo}Motion trail of a \textit{Crazyflie 2.1} quadcopter localizing and circumnavigating a stationary unknown target using the proposed control algorithms.}
\end{figure}

Early work in \cite{deghat2010Target,deghat2014Localization} introduced an estimator-coupled control framework, featuring a target estimator and a circumnavigation controller. The estimator adaptively generates an estimate of the target position using the agent's on-board bearing measurements and the agent's own position in an initial reference frame, whereas the controller enables the agent to orbit the target at a specified radius. This framework was later extended to localize multiple targets \cite{deghat2015Multitarget,chen2023Cooperative}, achieve unbiased estimation of a constant velocity target \cite{sui2024Unbiased}, adapt to 3D workspace \cite{li2018Localization,ma2023Finitetime}, accommodate discrete-time measurements \cite{chen2022DiscreteTime,wang2024Target}, include robustness to bounded load disturbances on the control inputs \cite{greiff2022Target} and to wind disturbances \cite{sui2024Adaptive}, remove the need for the agent's position \cite{wang2022Target,hu2022BearingOnly}, guide non-holonomic \cite{deghat2012Target} and double-integrator agents \cite{cao2023Bearingonly}, and extend to multi-agent systems \cite{dou2020Target,chen2022Multicircular,sui2024CollisionFree}.

With only a few exceptions, the studies mentioned above achieve at best an exponentially fast convergence rate, largely due to their reliance on the persistence-of-excitation condition. Notably, \cite{chen2023Finitetime} and \cite{ma2023Finitetime} are among the first to develop a finite-time stable target estimator using only local on-board bearing measurements. In a different vein, authors of \cite{zhou2023FiniteTime} and \cite{zhao2023Bearing} proposed distributed estimation algorithms using bearing measurements available in the local network to attain finite-time convergence. However, finite-time stability often results in a settling time that is an unbounded function of the system’s initial conditions, with an overly conservative theoretical estimate. These challenges motivate this paper to explore the BoTLC problem with predefined-time convergence. The main contributions are summarized as follows.

\begin{itemize}
    \item This paper presents a novel target estimation algorithm and an enhanced circumnavigation controller to solve the BoTLC problem. Unlike prior methods \cite{deghat2010Target,deghat2014Localization,deghat2015Multitarget,chen2023Cooperative,li2018Localization,ma2023Finitetime,chen2022DiscreteTime,wang2024Target,
    greiff2022Target,wang2022Target,hu2022BearingOnly,deghat2012Target,cao2023Bearingonly,
    dou2020Target,chen2022Multicircular,sui2024CollisionFree,sui2024Unbiased,chen2023Finitetime,zhou2023FiniteTime,zhao2023Bearing,sui2024Adaptive} which achieve finite-time convergence at best, our approach guarantees predefined-time convergence with a settling time that is independent of initial conditions, tunable \textit{a priori}, and tightly estimated. 

    \item The proposed control algorithms are validated through simulations and experimental tests, showing significantly faster convergence compared to existing techniques.
\end{itemize}

\section{PRELIMINARIES}
\label{sec:Preliminaries}
        
    \subsection{Useful Lemmas}
    \label{sec:useful_lemmas}
    
        We first summarize some technical lemmas that are conducive to the forthcoming stability analysis.

        \begin{definition}[see for instance, {\cite{chowdhary2013Concurrent,pan2018Composite,aranovskiy2023PreservingExcitation}}]
            A bounded signal $\bm{\phi}:\, \mathbb{R}_{+}\mapsto \mathbb{R}^{n\times \ell}$ is said to be \textit{interval exciting (IE)} if there exists $t_0\geq 0$, $T>0$, and $\mu>0$ such that 
            \begin{equation}
                \label{eq:def_interval_exciting}
                \int_{t_0}^{t_0+T}\bm{\phi}(\tau)\bm{\phi}^\top(\tau)\; \mathrm{d}\tau \geq \mu I,
            \end{equation}
            where $I$ represents the identity matrix of appropriate dimensions. Further, if \eqref{eq:def_interval_exciting} holds for all $t_0\geq 0$, then signal $\bm{\phi}(t)$ is said to be \textit{persistently exciting (PE)}.
        \end{definition}

        \begin{definition}[{\cite[Definition~2.4]{sanchez-torres2018Class}}]
            \label{definition:finite_time_stability}
            Consider the dynamical system described by
            \begin{equation}
                \label{eq:sample_dynamic_system}
                \dot{\bm{x}} = \bm{f}(t,\bm{x}; \bm{\rho}), 
            \end{equation}
            where $\bm{x}\in\mathbb{R}^n$ is the system state, $t\in[t_0,\infty)$ is the time variable with $t_0\in\mathbb{R}_{+}\cup \{0\}$, $\bm{\rho}\in\mathbb{R}^m$ is a constant vector representing the control parameters of the system, and $\bm{f}:\mathbb{R}_{+}\times \mathbb{R}^n$ is a nonlinear function, which can be discontinuous, and the solutions of \eqref{eq:sample_dynamic_system} are understood in the sense of Filippov \cite{filippov1988Differential}. Then a non-empty set $M\subset \mathbb{R}^n$ is said to be \textit{Globally Strongly Predefined-Time Attractive (GSPTA)} if any solution $\bm{x}(t,\bm{x}_0)$ of \eqref{eq:sample_dynamic_system} reaches $M$ in some finite time $t=t_0+T(\bm{x}_0)$, where the settling-time function $T:\mathbb{R}^n\mapsto \mathbb{R}$ is such that $\sup_{\bm{x}_0\in\mathbb{R}^n} T(\bm{x}_0) = T_c$, where $T_c$ is called the \textit{strong predefined time}.
        \end{definition}

        \begin{lemma}[{\cite[Theorem~2.2]{sanchez-torres2018Class}}]
            \label{lemma:PDT_stability_theorem}
            If there exists a continuous radially unbounded function $V:\mathbb{R}^n\mapsto \mathbb{R}_{+}\cup \{0\}$ such that $\bm{x}\in M$ iff $V(\bm{x})=0$ and any solution $\bm{x}(t)$ of system \eqref{eq:sample_dynamic_system} satisfies 
            \begin{equation}
                \label{eq:PDT_required_V_dot}
                \dot{V} = - \frac{1}{pT_c} \exp (V^p) V^{1-p}
            \end{equation}
            for constants $T_c = T_c(\bm{\rho})>0$ and $p\in (0,1]$, then the set $M$ is GSPTA for system \eqref{eq:sample_dynamic_system} and the strong predefined time is $T_c$.
        \end{lemma}

    \subsection{Background and Notations}
    \label{sec:background_and_notations}
        
        Consider a 2D plane containing a stationary target and an agent tasked with the BoTLC mission. The position of the target at time $t$, which is \textit{unknown} to the agent, is represented by $\bm{x}(t)\in\mathbb{R}^2$. The agent has a single-integrator model, 
        \begin{equation}
            \label{eq:agent_kinematics}
            \dot{\bm{y}}(t) = \bm{u}(t) 
        \end{equation}
        where $\bm{u}(t)$ is the control input to be designed. The agent's bearing measurement to the target is represented by the unit vector $\bm{\varphi}(t)$, defined as,
        \begin{equation}
            \label{eq:def_varphi}
            \bm{\varphi}(t)=\frac{\bm{x}(t)-\bm{y}(t)}{\norm{\bm{x}(t)-\bm{y}(t)}} =: \frac{\bm{x}(t)-\bm{y}(t)}{d(t)}
        \end{equation}
        where $d(t)$ denotes the Euclidean distance between the agent and the target at time $t\geq 0$. Let $\bar{\bm{\varphi}}(t)\in\mathbb{R}^2$ be a unit vector perpendicular to $\bm{\varphi}(t)$, obtained by $\pi/2$ clockwise rotation of $\bm{\varphi}(t)$. The bearing angle $\theta(t)$ is defined as the angle between the $x$-axis of the agent's local frame to the unit vector $\bm{\varphi}(t)$, with counterclockwise angles considered positive, as graphically illustrated in Fig.~\ref{fig:notation_conventions}. Other symbols $\hat{d}(t)$, $\hat{\bm{x}}(t)$, $\tilde{\bm{x}}(t)$, $d^*$, $\chi(t)$, $\gamma(t)$, and $\bm{\nu}$ will be explained in subsequent sections.

        \vspace{-0.5cm}

        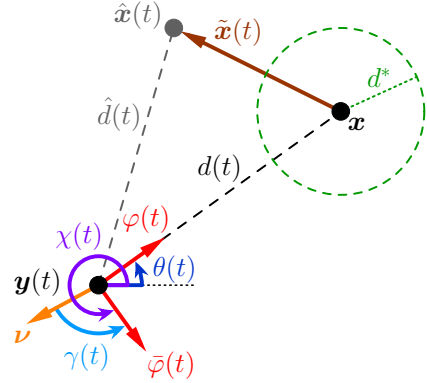
\begin{figure}[h]
            \centering
            \scalebox{0.74}{
                \begin{tikzpicture}[scale = 1.5, line cap=round,line join=round,>=triangle 45,x=1cm,y=1cm]
                    \clip(-4,-3.4) rectangle (1.2,1.56);
                    %
                    %
                    \draw [line width=1pt,dash pattern=on 5pt off 5pt] (-2.9019083532629173,-2.0826434806438683)-- (0,0);
                    \draw (-1.8,-0.4293090608150226) node[anchor=north west] {\Large $d(t)$};
                    %
                    %
                    \draw [line width=1pt,dash pattern=on 5pt off 5pt,color=wqwqwq] (-2.9019083532629173,-2.0826434806438683)-- (-2,1);
                    \draw [color=wqwqwq](-3,0.25) node[anchor=north west] {\Large $\hat{d}(t)$};
                    
                    %
                    %
                    \draw [-{Triangle[length=5mm,width=2mm]},line width=2pt,color=ffqqqq] (-2.9019083532629173,-2.0826434806438683) -- (-2.0894814568450726,-1.499580415464064);
                    \draw [color=ffqqqq](-2.7,-1.05) node[anchor=north west] {\Large $\varphi(t)$};
                    %
                    %
                    \draw [-{Triangle[length=5mm,width=2mm]},line width=2pt,color=ffqqqq] (-2.9019083532629173,-2.0826434806438683) -- (-2.318845288083113,-2.895070377061713);
                    \draw [color=ffqqqq](-2.4,-2.8) node[anchor=north west] {\Large $\bar{\varphi}(t)$};
                    %
                    %
                    \draw [line width=0.8pt,dotted] (-2.9019083532629173,-2.0826434806438683)-- (-1.74573,-2.08264);
        
                    \draw [
                        shift={(-2.9019083532629173,-2.0826434806438683)},
                        -{Stealth[bend, length=3mm, width=2.5mm]}, 
                        line width=2pt,
                        color=qqttcc
                    ] (0.00017248740479756843:0.5202800621839034) arc (0.00017248740479756843:33:0.5202800621839034);
                    \draw [shift={(-2.9019083532629173,-2.0826434806438683)},line width=2pt, color=qqttcc] (0,0) -- (0:0.5202800621839034);
                    \draw [color=qqttcc](-2.3353811744404442,-1.65) node[anchor=north west] {\Large $\theta(t)$};
        
                    \draw [
                        shift={(-2.9019083532629173,-2.0826434806438683)},
                        -{Stealth[bend, length=3mm, width=2.5mm]}, 
                        line width=2pt,
                        color=qqzzff
                    ] (-151.69981639350945:0.578088957982115) arc (-151.69981639350945:-57:0.578088957982115);
                    \draw [color=qqzzff](-3.3990648571275357,-2.7) node[anchor=north west] {\Large $\gamma(t)$};
        
                    \draw [-{Triangle[length=5mm,width=2mm]},line width=2pt,color=ffxfqq](-2.9019083532629173,-2.0826434806438683) -- (-3.7823841875339776,-2.556734511211501);
                    \draw [color=ffxfqq](-4,-2.55) node[anchor=north west] {\Large $\bm{\nu}$};
        
                    \draw [
                        shift={(-2.9019083532629173,-2.0826434806438683)},
                        -{Stealth[bend, length=3mm, width=2.5mm]}, 
                        line width=2pt,
                        color=xfqqff
                    ] (0.00017248740479756843:0.346853374789269) arc (0.00017248740479756843:302:0.346853374789269);
                    \draw [shift={(-2.9019083532629173,-2.0826434806438683)},line width=2pt, color=xfqqff] (0,0) -- (0:0.346853374789269);
                    \draw [color=xfqqff](-3.5,-1.25) node[anchor=north west] {\Large $\chi(t)$};

                    \draw [fill=black] (-2.9019083532629173,-2.0826434806438683) circle (3pt);
                    \draw (-4,-1.8) node[anchor=north west] {\Large $\bm{y}(t)$};

                    \draw [-{Triangle[length=5mm,width=2mm]},line width=2pt,color=zzttqq] (0,0) -- (-1.94, 0.97);
                    \draw [color=zzttqq](-1.6,1.2) node[anchor=north west] {\Large $\tilde{\bm{x}}(t)$};
        
                    \draw [line width=1pt,dash pattern=on 4pt off 4pt,color=qqzzqq] (0,0) circle (1cm);
                    \draw [line width=1pt,dotted,color=qqzzqq] (0,0)-- (0.9097994775058879,0.41504808243143776);
                    \draw [color=qqzzqq](0.23133379900014597,0.6) node[anchor=north west] {\Large $d^*$};
        
                    \draw [fill=black] (0,0) circle (3pt);
                    \draw (0,0) node[anchor=north west] {\Large $\bm{x}$};

                    \draw [draw = wqwqwq, fill=wqwqwq] (-2,1) circle (3pt);
                    \draw [color=wqwqwq](-2.8,1.4) node[anchor=north west] {\Large $\hat{\bm{x}}(t)$};
        


                \end{tikzpicture}
             }
            \caption{\label{fig:notation_conventions}Problem geometry and notations.}
        \end{figure}
        
    \subsection{Problem Statements}
    \label{sec:problem_statements}
        \begin{problem}[Target Localization]
            \label{problem:Localization_Problem}
            \normalfont Design a target estimator $\dot{\hat{\bm{x}}}(t)=\dot{\hat{\bm{x}}} (\bm{y}(t),\bm{\varphi}(t))$ such that the agent localizes a stationary target $\bm{x}$ using only local bearing measurements $\bm{\varphi}$ and the agent's own location $\bm{y}$. In particular, the estimation error $\tilde{\bm{x}}(t)$, defined below, should converge to the origin with a tunable strong predefined time $T_{c,1}> 0$:
            \begin{equation}
                \label{eq:def_x_tilde}
                \lim_{t \to T_{c,1}} \norm{\tilde{\bm{x}}(t)} = \lim_{t \to T_{c,1}} \norm{\hat{\bm{x}}(t) - \bm{x}} = 0 .
            \end{equation}
        \end{problem}

        \vspace{0.25cm}
        
        \begin{problem}[Target Circumnavigation]
            \label{problem:Circumnavigation_Problem}
            \normalfont Design a circumnavigation controller $\bm{u}(t)=\bm{u}(\bm{y}(t),\bm{\varphi}(t),\hat{\bm{x}}(t),d^*,\omega^*)$ such that after some tunable strong predefined time $T_{c,2}>0$, the agent converges to a circular orbit centered at the target with the desired radius $d^{*}>0$, that is,
            \begin{equation}
                \label{eq:def_tracking_error}
                \lim_{t\to T_{c,2}} \norm{\bm{x}  - \bm{y} (t)} - d^{*} = 0,
            \end{equation}
            while ensuring sustained motion along this circular orbit, that is, $\dot{\theta}(t)=\omega^*$ for all $t> T_{c,2}$ where $\omega^*>0$ is the desired angular velocity.
                
        \end{problem}

\section{PROPOSED ALGORITHM}
\label{sec:proposed_algorithm}
    
    \subsection{Target Estimator}
    \label{sec:Target Estimator}
        To solve \textit{Problem~\ref{problem:Localization_Problem}}, we propose the following predefined-time target position estimator:
        \begin{equation}
            \label{eq:def_target_estimator}
            \dot{\hat{\bm{x}}}(t) = -\frac{1}{\alpha_1 T_{c,1}} \exp(\norm{\bm{\xi}(t)}^{\alpha_1}) \bm{\psi}^{\alpha_1}\big(\bm{\xi}(t)\big), 
        \end{equation}
        where $\alpha_1 \in (0,1]$ is a control constant to be chosen, $T_{c,1}>0$ is the tunable strong predefined time, $\bm{\psi}^{\beta}(\bm{z}):=\norm{\bm{z}}^{-\beta}\bm{z}$ if $\bm{z}\neq \bm{0}_n$, and $\bm{\psi}^{\beta}(\bm{z}):=\bm{0}_n$ if $\bm{z}=\bm{0}_n$ for some $\beta \geq 0$ and $\bm{z}\in\mathbb{R}^n$, and $\bm{\xi}(t)\in\mathbb{R}^2$ is the proposed reconstructed error signal, given by
        \begin{equation}
            \label{eq:def_xi}
            \bm{\xi}(t) = \begin{cases}
                P^{-1}(t)\big(P(t)\hat{\bm{x}}(t)-\bm{q}(t)\big), \quad & t > 0, \\
                \bm{0}_2 & t=0,
            \end{cases}
        \end{equation}
        which incorporates Kreisselmeier’s regressors $P(t) \in \mathbb{R}^{2\times 2}$ and $\bm{q} (t) \in \mathbb{R}^2$ that were first introduced to the BoTLC problem in \cite{chen2022Multicircular,chen2023Finitetime} and are defined as follows,
        \begin{subequations}
            \label{eq:def_aux_regressors}
            \begin{alignat}{2}
                \dot{P} (t) &= -P (t) + \bar{\bm{\varphi}}(t) \bar{\bm{\varphi}}^\top (t), && P (0) = \bm{0}_{2\times 2},\label{eq:P_dot_def} \\
                \dot{\bm{q}}(t) &= -\bm{q} (t) + \bar{\bm{\varphi}}(t) \bar{\bm{\varphi}}^\top (t) \bm{y} (t), \quad &&\bm{q}(0) = \bm{0}_2, \label{eq:q_dot_def}
            \end{alignat}
        \end{subequations}
        where $\bm{0}_{2\times 2}$ denotes a 2-by-2 zero matrix, and $\bm{0}_2$ represents a 2-dimensional zero vector.

    \subsection{Circumnavigation Controller}
    \label{sec:circumnavigation_control_law}
        We propose the following circumnavigation control law to answer \textit{Problem~\ref{problem:Circumnavigation_Problem}},
        \begin{align}
            \label{eq:def_circum_controller}
            \begin{split}
                \bm{u}(t) &= \frac{\exp\big(|\tilde{d}(t)|^{\alpha_2}\big)}{\alpha_2 T_{c,2}} \sig^{1-\alpha_2}\big( \tilde{d}(t)\big) \bm{\varphi}(t)+ k_\omega \bar{\bm{\varphi}}(t),
            \end{split}
        \end{align}
        where $\alpha_2 \in (0,1], k_\omega =\omega^*d^*$ are control gains, $\sig^\alpha(\bm{z}):=[\sign(z_1)|z_1|^\alpha, \cdots, \sign(z_n)|z_n|^\alpha]^\top\in\mathbb{R}^n$ for some vector $\bm{z}=[z_1,\cdots,z_n]^\top\in\mathbb{R}^n$ and constant $\alpha>0$ with $\sign(\cdot)$ being the standard signum function, $T_{c,2}>T_{c,1}>0$ is the tunable strong predefined time, and $\tilde{d}(t)=\hat{d}(t)-d^*$ is the error between the desired circumnavigation radius $d^*$ and the agent's estimate of the agent-to-target distance $\hat{d}(t)$, which is defined as,
        \begin{equation}
            \label{eq:def_d_hat}
            \hat{d}(t):=||\hat{\bm{x}}(t)-\bm{y}(t)||.
        \end{equation}

\section{STABILITY ANALYSIS}
\label{sec:stability_analysis}
    This section demonstrates that the proposed target estimator \eqref{eq:def_target_estimator} and controller \eqref{eq:def_circum_controller} jointly solve \textit{Problem~\ref{problem:Localization_Problem}} and \textit{Problem~\ref{problem:Circumnavigation_Problem}}. The proof is organized into three parts: Section~\ref{sec:Stability_Proof_Part_A} establishes preliminary results that hold for $t \in [0, \mathcal{T}_1)$, where $\mathcal{T}_1 > 0$ is detailed in \textbf{Lemma~\ref{lemma:d_is_bounded_by_IVT}}. Section~\ref{sec:Generalization} then extends these results to all \( t \geq 0 \). Finally, the main results are presented in Section~\ref{sec:Main_Results}.

    Before proceeding, the following assumptions will be maintained throughout this paper.

    \begin{assumption}
        \label{assumption:d(0)_initial_condition}
        We assume that the agent and the target occupy different locations at $t=0$, that is, $d(0)>0$. 
    \end{assumption}

    \begin{assumption}
        \label{assumption:stationary_target}
        The target is assumed to remain stationary throughout the system evolution, that is, $\dot{\bm{x}}(t)\equiv \bm{0}_2$ for all $t\geq 0$.
    \end{assumption}

    \begin{assumption}
        \label{assumption:estimation_error}
        We suppose that the norm of the initial target estimation error $\norm{\tilde{\bm{x}}(0)}$ is small such that there exists some positive constant $\eta$ satisfying
        \begin{equation}
            \label{eq:d_star_condition}
            d^* - \norm{\tilde{\bm{x}}(0)} \geq \eta > 0.
        \end{equation}
    \end{assumption}

    \subsection{Preliminary Results}
    \label{sec:Stability_Proof_Part_A}

    To facilitate the following lemma, define set $\mathcal{D}$ as
    \begin{equation}
        \label{eq:def_set_D}
        \mathcal{D}:= \big\{ d \, | \, d_{min} \leq d(t) \leq d_{max}\big\}
    \end{equation}
    where $d_{max}$ is a constant defined as
        \begin{equation}
            \label{eq:def_d_max}
            d_{max} = \max( 2d^*-d_{s} , d(0)) + d_{\varpi},
        \end{equation}
    with $d_{min},d_s,d_\varpi>0$ being any constants chosen to meet the below inequality,
    \begin{equation}
        \label{eq:def of d_min, d_s, d_varpi}
        0 < d_\varpi < d_{min} < d_s < \eta .
    \end{equation}
    Here, $d_\varpi$ and $d_s$ are used solely in the stability analysis and do not bear physical meaning. As will be established later in \textbf{Lemma~\ref{lemma:d_is_bounded_for_all_time}}, $d_{min}$ and $d_{max}$ represents the lower and upper bound of $d(t)$, respectively (resp. hereafter).
    
    \begin{lemma}
        \label{lemma:d_is_bounded_by_IVT}
        Under the target estimator \eqref{eq:def_target_estimator} and the circumnavigation controller \eqref{eq:def_circum_controller}, there exists some time $\mathcal{T}_1 >0$ such that 
        \begin{equation}
            \label{eq:d_is_bounded_by_IVT}
            d(t) \in \mathcal{D}, \quad t\in [0,\mathcal{T}_1).
        \end{equation}
    \end{lemma}

    \begin{proof}
        Since the agent has a single integrator kinematics \eqref{eq:agent_kinematics}, and $\bm{u}(t)$ given in \eqref{eq:def_circum_controller} is continuous in time, $d(t)$ is continuous in time. Additionally, from the definition of $d_{min}$ given in \eqref{eq:def of d_min, d_s, d_varpi} ($d_{max}$ in \eqref{eq:def_d_max}, resp.), we have $d(0)-d_{min}>0$ ($d_{max}-d(0)>0$, resp.) for all $d(0)\in\mathbb{R}\setminus \{0\}$. Therefore, for any $\epsilon >0$, there exists a $\delta>0$ such that $t<\delta \implies |d(t)-d(0)|<\epsilon$. Choose $\epsilon = 1/2 \min \big(d(0)-d_{min},d_{max}-d(0)\big)$, then for any $0\leq t < \delta$, we have $d(t)\in(d_{min},d_{max})$. This proof is completed by setting $\delta = \mathcal{T}_1$.
    \end{proof}

    \begin{remark}
        Because both the target estimator \eqref{eq:def_target_estimator} and the controller \eqref{eq:def_circum_controller} include the term $\bm{\varphi}(t)$, the term $\bm{\varphi}(t)$ should be well-defined. \textbf{Lemma~\ref{lemma:d_is_bounded_by_IVT}} provides a sufficient condition ensuring that $d(t)\neq 0$, and therefore, $\bm{\varphi}(t)$ is well-defined for $t\in [0,\mathcal{T}_1)$. Later in \textbf{Lemma~\ref{lemma:results_generalization}}, results in \textbf{Lemma~\ref{lemma:d_is_bounded_by_IVT}} will be extended to hold for all $t\geq 0$.
    \end{remark}

    \begin{lemma}
        \label{lemma:IE_by_IVT}
        For any $t_0 \geq 0$ and $T>0$ such that $[0,t_0+T] \subset [0,\mathcal{T}_1)$, the signal $\bar{\bm{\varphi}}(t)$ is IE over $[0,t_0+T]$.
    \end{lemma}
    
    \begin{proof}
        This proof follows the same ideas as in \cite[Lemma~3]{deghat2010Target}. Let $\bm{\nu}\in\mathbb{R}^2$ be some constant unit vector representing a fixed direction in the world frame. Then, according to \cite[Section 2.5]{sastry1994Adaptive}, condition \eqref{eq:def_interval_exciting} can be expressed in scalar form such that the signal $\bar{\bm\varphi}(t)$ is IE if there exists $\mu>0$ such that
        \begin{equation}
            \label{eq:exciting_scalar_form}
            \mu \leq\int_{t_0}^{t_0+T} \big(\bm{\nu}^\top \bar{\bm\varphi}(\tau)\big)^2 \; \mathrm{d}\tau
        \end{equation}
        holds for all constant unit vectors $\bm{\nu}$. Let $\gamma(t)$ be the angle measured counter-clockwise from the unit vector $\bm{\nu}$ to the vector $\bar{\bm\varphi}(t)$. See Fig.~\ref{fig:notation_conventions} for a sample vector $\bm\nu$ and the corresponding $\gamma(t)$ angle. Subsequently, \eqref{eq:exciting_scalar_form} can be written in terms of the angle $\gamma(t)$ as 
        \begin{equation}
            \label{eq:exciting_angle}
            \mu \leq \int_{t_0}^{t_0+T}\cos^2\gamma(\tau) \; \mathrm{d}\tau .
        \end{equation}
        Noticing that the angle $\chi(t)-\gamma(t)$ is always constant as $\bm{\nu}$ is a constant unit vector and that $\bar{\bm{\varphi}}(t)\perp \bm{\varphi}(t)$ for all $t\geq 0$, we obtain that
        \begin{equation}
            \frac{\mathrm{d}\gamma (t)}{\mathrm{d}t} = \frac{\mathrm{d}\chi(t)}{\mathrm{d}t} = \frac{\mathrm{d}\theta(t)}{\mathrm{d}t}. \label{eq:Dzeta}
        \end{equation}
        Since $\mathrm{d}\theta(t)/\mathrm{d}t$ represents the angular velocity of the agent relative to the target, we can express $\mathrm{d}\gamma(t)/\mathrm{d}t$ as
        \begin{align*}
            \scalemath{0.85}{\frac{\mathrm{d}\gamma(t)}{\mathrm{d}t}} \overset{\eqref{eq:Dzeta}}&{=} \scalemath{0.85}{\frac{\bar{\bm{\varphi}}^\top (t)\big(\dot{\bm{y}} - \dot{\bm{x}}\big)}{d(t)}} \\
            \overset{\textit{Asm.\ref{assumption:stationary_target}},\eqref{eq:agent_kinematics},\eqref{eq:def_circum_controller}}&{=} \scalemath{0.8}{\frac{1}{d(t)} \bar{\bm{\varphi}}^\top (t) \Bigg( \frac{\exp\big(|\tilde{d}(t)|^{\alpha_2}\big)}{\alpha_2 T_{c,2}} \sig^{1-\alpha_2}\big( \tilde{d}(t)\big) \bm{\varphi}(t)+k_{\omega} \bar{\bm{\varphi}}(t) \Bigg)} \\
            \overset{\bar{\bm{\varphi}}^\top \bm{\varphi} = 0}&{=} \scalemath{0.85}{\frac{1}{d(t)} k_\omega \bar{\bm{\varphi}}^\top (t) \bar{\bm{\varphi}}(t)  \overset{\bar{\bm{\varphi}}^\top \bar{\bm{\varphi}}=1}{=}\frac{k_\omega}{d(t)} .} \numberthis \label{eq:pe_bound_step1}
        \end{align*}
        From \textbf{Lemma~\ref{lemma:d_is_bounded_by_IVT}}, we know that $d(t)$ is bounded over the time interval $[0,\mathcal{T}_1)$. It then immediately follows that
        \begin{equation}
            \label{eq:dgamma/dt>0}
            \frac{\mathrm{d}\gamma(t)}{\mathrm{d}t} \overset{\eqref{eq:d_is_bounded_by_IVT}}{\geq} \frac{k_\omega}{d_{max}} > 0 .
        \end{equation}
        Consequently, we have
        \begin{equation}
            \gamma(t_0+t) \geq \gamma(t_0)+\frac{k_\omega t}{d_{max}},
        \end{equation}
        for $t\in [0,t_0+T] \subset [0,\mathcal{T}_1)$. Therefore, the continuous function $\gamma(t)$ is monotonically increasing over the time interval $[0, t_0+T]$ and does not converge to a constant value. Hence, it is always possible to find some positive constant $\mu$ and $T$ that satisfy \eqref{eq:exciting_angle}.
    \end{proof}

    \begin{lemma}
        \label{lemma:P(t)_is_invertible_IVT}
        The Kreisselmeier's regressor matrix $P(t)$ defined in \eqref{eq:P_dot_def} is non-singular over the time interval $(0,\mathcal{T}_1)$.
    \end{lemma}

    \begin{proof}
        Let the signal $\Delta : \mathbb{R}_{+}\mapsto \mathbb{R}$ be the determinant of $P(t)$. According to \cite[Proposition~1]{aranovskiy2023PreservingExcitation}, if the signal $\bar{\bm{\varphi}}(t)$ is IE over the time interval $[0,t_0+T]\subset [0,\mathcal{T}_1)$, then the signal $\Delta(t)$ is also IE over the same time interval. Since \textbf{Lemma~\ref{lemma:IE_by_IVT}} establishes that $\bar{\bm{\varphi}}(t)$ is IE, it follows immediately that $\Delta(t)$ is IE, which implies that $P(t)$ is non-singular over the time interval $(0,t_0+T]\subset [0,\mathcal{T}_1)$.
    \end{proof}

    To bound the target estimation error $||\tilde{\bm{x}}(t)||$ during $t\in[0,\mathcal{T}_1)$, we first present some useful propositions. 

    \begin{proposition}
        For all $t\geq 0$, it holds that
        \begin{equation}
            \label{eq:Px=q}
            P(t)\bm{x} = \bm{q}(t).
        \end{equation}
    \end{proposition}

    \begin{proof}
        This proof follows the same ideas as in \cite[Lemma~1]{chen2022Multicircular}. Introduce an auxiliary function $\bm{f}(t):= P(t)\bm{x}-\bm{q}(t)$, whose time derivative is obtained as
        \begin{align*}
            \dot{\bm{f}}(t) &= \dot{P}(t)\bm{x} + P(t)\dot{\bm{x}} - \dot{\bm{q}}(t) \\
            \overset{\eqref{eq:def_aux_regressors}}&{=} \big(-P(t)+\bar{\bm{\varphi}}(t)\bar{\bm{\varphi}}^\top(t)\big)\bm{x} +P(t)\dot{\bm{x}} \\
            &\qquad\quad -\big( -\bm{q}(t) + \bar{\bm{\varphi}}(t)\bar{\bm{\varphi}}^\top(t) \bm{y}(t) \big) \\
            \overset{\textit{Asm.~\ref{assumption:stationary_target}}}&{=} - \big(P(t)\bm{x}-\bm{q}(t)\big) + \bar{\bm{\varphi}}(t)\bar{\bm{\varphi}}^\top(t) \big(\bm{x}-\bm{y}(t)\big) \\
            \overset{\eqref{eq:def_varphi}}&{=} -\bm{f}(t) + \bar{\bm{\varphi}}(t)\bar{\bm{\varphi}}^\top(t) \bm{\varphi}(t) \norm{\bm{x}-\bm{y}(t)} \\
            \overset{\bar{\bm{\varphi}}^\top\bm{\varphi}=0}&{=} -\bm{f}(t) . \numberthis \label{eq:f_dot=-f}
        \end{align*}
        Given \eqref{eq:f_dot=-f} and noticing that $\bm{f}(0)=P(0)\bm{x} - \bm{q}(0)\overset{\eqref{eq:def_aux_regressors}}{=}\bm{0}_2$, we conclude that $\bm{f}(t)\equiv \bm{0}_2$ for all $t\geq 0$.
    \end{proof}

    \begin{proposition}
        \label{proposition:xi=x_tilde_IVT}
        For $t \in (0,t_0+T]\subset [0,\mathcal{T}_1)$, it holds that 
        \begin{equation}
            \label{eq:xi=x_tilde_IVT}
            \bm{\xi}(t) = \tilde{\bm{x}}(t).
        \end{equation}
    \end{proposition}

    \begin{proof}
        From \textbf{Lemma~\ref{lemma:P(t)_is_invertible_IVT}}, we know that $P(t)$ is non-singular during the time interval $(0,t_0+T]\subset [0,\mathcal{T}_1)$. Hence, the reconstructed estimation error signal $\bm{\xi}(t)$ can be expressed as,
        \begin{align*}
            \bm{\xi}(t) \overset{\eqref{eq:def_xi}}&{=} P^{-1}(t)\big(P(t)\hat{\bm{x}}(t)-\bm{q}(t)\big) \\
            \overset{\eqref{eq:Px=q}}&{=} P^{-1}(t)\big(P(t)\hat{\bm{x}}(t)-P(t)\bm{x}\big) \\
            &= P^{-1}(t) P(t) \tilde{\bm{x}}(t) \\
            \overset{\textit{Lem.~\ref{lemma:P(t)_is_invertible_IVT}}}&{=} \tilde{\bm{x}}(t), \qquad t \in (0, t_0+T]\subset [0, \mathcal{T}_1).
        \end{align*}
    \end{proof}

    We now provide a lemma ensuring that the target estimation error $||\tilde{\bm{x}}(t)||$ is bounded during the time interval $[0,\mathcal{T}_1)$.

    \begin{lemma}
        \label{lemma:x_tilde_is_bounded_by_IVT}
        Under the target estimator \eqref{eq:def_target_estimator} and the circumnavigation controller \eqref{eq:def_circum_controller}, the norm of the target estimation error $\norm{\tilde{\bm{x}}(t)}$ is bounded by
        \begin{equation}
            \label{eq:x_tilde_is_bounded_by_IVT}
            \norm{\tilde{\bm{x}}(t)} \leq \norm{\tilde{\bm{x}}(0)}, \qquad t \in [0,\mathcal{T}_1).
        \end{equation}
    \end{lemma}

    \begin{proof}
        Consider the Lyapunov candidate $V_0=1/2\tilde{\bm{x}}^\top(t)\tilde{\bm{x}}(t)$, whose time derivative along the system trajectories is found as
        \begin{align*}
            \dot{V}_0 \overset{\eqref{eq:def_target_estimator}}&{=} \tilde{\bm{x}}^\top(t) \left(-\frac{1}{\alpha_1 T_{c,1}} \exp(\norm{\bm{\xi}(t)}^{\alpha_1})\bm{\psi}^{\alpha_1}\big(\bm{\xi}(t)\big)\right) \\[3pt]
            \overset{\eqref{eq:xi=x_tilde_IVT}}&{=} - \frac{\exp(\norm{\tilde{\bm{x}}(t)}^{\alpha_1}) }{\alpha_1 T_{c,1}} \tilde{\bm{x}}^\top(t) \frac{\tilde{\bm{x}}(t)}{\norm{\tilde{\bm{x}}(t)}^{\alpha_1}} \\[3pt]
            &= - \frac{\exp(\norm{\tilde{\bm{x}}(t)}^{\alpha_1}) }{\alpha_1 T_{c,1}} \norm{\tilde{\bm{x}}(t)}^{2-\alpha_1} ,
        \end{align*} 
        which is negative definite. It follows that the bound \eqref{eq:x_tilde_is_bounded_by_IVT} holds.
    \end{proof}

    With the above preparation, we are now ready to generalize the results to all $t\geq 0$.
    
    \subsection{Generalization}
    \label{sec:Generalization}

    \begin{lemma}
        \label{lemma:d_is_bounded_for_all_time}
        Under the target estimator \eqref{eq:def_target_estimator} and the circumnavigation controller \eqref{eq:def_circum_controller}, \textbf{Lemma~\ref{lemma:d_is_bounded_by_IVT}} holds for all $t\geq 0$.
    \end{lemma}

    \begin{proof}
        \label{proof:d_is_bounded_for_all_t}
        This proof employs a similar approach to \cite[Lemma~4.7]{sui2024Unbiased}, and is divided into two parts: in Part I, we derive the dynamics of $d(t)$ as preparation; in Part II, proof by contradiction is employed to establish the conclusion. 
        
        \textbf{Part I.} Using the triangle inequality, we obtain that,
        \begin{align*}
            \norm{\tilde{\bm{x}}(t) } \geq \big|d (t) - \hat{d} (t) \big|. \numberthis \label{eq:triangular_inequality}
        \end{align*}
        Introduce auxiliary variables $\delta(t)$ and $\varrho(t)$ as
        \begin{subequations}
            \label{eq:def_aux_delta_and_varrho}
            \begin{align*}
                \delta(t) &= d(t)- d^* ,\numberthis \label{eq:def_aux_delta} \\
                \varrho(t) &= d(t)-\hat{d}(t) . \numberthis \label{eq:def_aux_varrho}
            \end{align*}
        \end{subequations}
        The inequality \eqref{eq:triangular_inequality} can thereby be written, using \textbf{Lemma~\ref{lemma:x_tilde_is_bounded_by_IVT}}, in the form of
        \begin{equation}
            \label{eq:varrho_is_bounded}
            |\varrho (t)| \leq \norm{\tilde{\bm{x}} (t)} \leq \norm{\tilde{\bm{x}}(0)}, \quad t \in [0,\mathcal{T}_1).
        \end{equation}
        The dynamics of $\delta(t)$ is obtained as
        \begin{align*}
            \scalemath{0.85}{\dot{\delta}(t)} &= \scalemath{0.85}{\dot{d}(t) - 0} \\[3pt]
            &= \scalemath{0.85}{\frac{\big(\bm{x}-\bm{y}(t)\big)^\top \big(\dot{\bm{x}}-\dot{\bm{y}}(t)\big)}{d(t)}} \overset{\eqref{eq:def_varphi}}{=} \scalemath{0.85}{\bm{\varphi}^\top(t) \big(\bm{0}_2 - \dot{\bm{y}}(t)\big)} \\[3pt]
            \overset{\eqref{eq:agent_kinematics},\eqref{eq:def_circum_controller}}&{=} \scalemath{0.85}{-\bm{\varphi}^\top(t) \Bigg(\frac{\exp\big(|\tilde{d}(t)|^{\alpha_2}\big)}{\alpha_2 T_{c,2}} \sig^{1-\alpha_2}\big( \tilde{d}(t)\big) \bm{\varphi}(t)+ k_\omega \bar{\bm{\varphi}}(t) \Bigg)} \\[3pt]
            \overset{\bm{\varphi}^\top\bar{\bm{\varphi}}=0}&{=} \scalemath{0.85}{ -\frac{\exp\big(|\tilde{d}(t)|^{\alpha_2}\big)}{\alpha_2 T_{c,2}} \sig^{1-\alpha_2}\big( \hat{d}(t) - d(t) + d(t) - d^*\big)} \\[3pt]
            \overset{\eqref{eq:def_aux_delta_and_varrho}}&{=} \scalemath{0.85}{\sign \big( -\delta(t) + \varrho(t) \big) \frac{\exp\big(|\tilde{d}(t)|^{\alpha_2}\big)}{\alpha_2 T_{c,2}} |\delta(t) - \varrho(t)|^{1-\alpha_2}.} \numberthis \label{eq:delta_dynamics}
        \end{align*}
        From \eqref{eq:delta_dynamics}, we know that the sign of $\dot{\delta}(t)$ depends solely on the sign of $-\delta(t)+\varrho(t)$.
        
        \textbf{Part II.} Suppose $\mathcal{T}_1$ specifies the first exit time from $\mathcal{D}$, that is, $\mathcal{T}_1 = \inf\limits_{t>0} d(t)\not\in\mathcal{D}$. Assume that the first exit time is finite, that is, $\mathcal{T}_1<\infty$. Due to the continuity of $d(t)$ in time, there must exist a finite time interval $t\in [\tau-\epsilon, \tau]\subseteq [0,\mathcal{T}_1]$ and some $\epsilon \geq \mathrm{d}t$ where $\mathrm{d}t$ is an infinitesimally small change in time, such that $\dot{d}(t)$ is either 
        \begin{enumerate}
            \item strictly negative if it were the case that $d(\mathcal{T}_1) < d_{min}$, or
            \item strictly positive if it were the case that $d(\mathcal{T}_1) > d_{max}$.
        \end{enumerate}
        We now consider the two cases in more detail and show that each case leads to a contradiction.
        \\
        \textbf{Case 1:} $d(\mathcal{T}_1) < d_{min}$. Suppose at time $t = \tau -\epsilon$, we have $\dot{d}(t) < 0$. Further, for the case where $d(t)$ were to reduce beyond $d_{min}$, it would be sufficient to consider the situation where $d(t)$ is already very close to the boundary $d_{min}$. In other words, we consider the situation where
        \begin{equation}
            \label{eq:delta_t_case_1_range}
            \delta(t) \in [ d_{min}-d^*, d_{min}-d^* + \underline{\varepsilon}], \quad t \in [\tau-\epsilon, \tau],
        \end{equation}
        with $\underline{\varepsilon}>0$ being some small constant satisfying
        \begin{equation}
            \label{eq:def_varepsilon_under}
            d_{min}-d^* + \underline{\varepsilon} < d_s-d^*.
        \end{equation}
        See Fig.~\ref{fig:delta_bound} for a graphical illustration of the depicted situation and the considered range of $\delta(t)$ given in \eqref{eq:delta_t_case_1_range}. 
        
        From \eqref{eq:delta_dynamics}, we have $\sign\big(\dot{\delta}(t)\big) = \sign\big(-\delta(t)+\varrho(t)\big)$, and from \eqref{eq:def_aux_delta}, it is clear that $\sign\big(\dot{d}(t)\big)=\sign\big(\dot{\delta}(t)\big)$. Therefore, to determine the sign of $\dot{d}$ at time $t' = \tau -\epsilon + \mathrm{d}t$, we just need to evaluate $-\delta(t')+\varrho(t')$:
        \begin{align*}
            -\delta(t'&)+\varrho(t') \\
            \overset{\eqref{eq:delta_t_case_1_range}}&{\geq} -(d_{min} - d^* + \underline{\varepsilon}) + \varrho(t') \\
            \overset{\eqref{eq:def_varepsilon_under}}&{\geq} -(d_s - d^*) + \varrho(t') \\
            &\geq-(d_s - d^*) - |\varrho(t')| \\
            \overset{\eqref{eq:varrho_is_bounded}}&{\geq} -(d_s - d^*) - \norm{\tilde{\bm{x}}(0)} \\
            &= d^*-d_s - \norm{\tilde{\bm{x}}(0)} \overset{\eqref{eq:d_star_condition}}{\geq} \eta -d_s \overset{\eqref{eq:def of d_min, d_s, d_varpi}}{>} 0,  \numberthis \label{eq:delta_pos_inequality_IVT} 
        \end{align*}
        which contradicts the implication that $\dot{d}(t')< 0$.
        \\
        \textbf{Case 2:} $d(\mathcal{T}_1) > d_{max}$. Similarly, suppose at time $t=\tau-\epsilon$, we have $\dot{d}(t) > 0$. Furthermore, for the case where $d(t)$ were to increase beyond $d_{max}$, it would be sufficient to consider the situation where $d(t)$ is very close to the boundary $d_{max}$, that is, we consider the situation where
        \begin{equation}
            \label{eq:delta_t_case_2_range}
            \delta(t) \in [d_{max}-d^* - \bar{\varepsilon}, d_{max} - d^* ],\quad t \in [\tau-\epsilon, \tau],
        \end{equation}
        with $\bar{\varepsilon}\in(0,d_{\varpi})$ being any constant chosen in the range. See Fig.~\ref{fig:delta_bound} for a graphical illustration of the depicted situation and the considered range of $\delta(t)$ given in \eqref{eq:delta_t_case_2_range}. 
        
        Likewise, the sign of $d(t')$ at time $t'=\tau-\epsilon + \mathrm{d}t$ can be evaluated using the term $-\delta(t')+\varrho(t')$, as follows,
        \begin{align*}
            -\delta(t')&+\varrho(t') \\
            \overset{\eqref{eq:delta_t_case_2_range}}&{\leq} -(d_{max}-d^*- \bar{\varepsilon}) + \varrho(t') \\
            &\leq -(d_{max}-d^*- \bar{\varepsilon}) + |\varrho(t')| \\
            \overset{\eqref{eq:varrho_is_bounded}}&{\leq} -(d_{max}-d^*- \bar{\varepsilon}) + \norm{\tilde{\bm{x}}(0)} \\
            &= d^* - d_{max} + \bar{\varepsilon} + \norm{\tilde{\bm{x}}(0)}. \numberthis \label{eq:delta_neg_inequality_general_case} 
        \end{align*}
        First consider the scenario where $2d^*-d_s \leq d(0)$, and therefore, $d_{max}=d(0)+d_{\varpi}$. Subsequently, \eqref{eq:delta_neg_inequality_general_case} can be expressed as,
        \begin{align*}
            -\delta(t')&+\varrho(t') \\
            &= d^* - \big(d(0) + d_{\varpi}\big) + \bar{\varepsilon} + \norm{\tilde{\bm{x}}(0)} \\
            &\leq d^* - (2d^*-d_s +d_\varpi) + \bar{\varepsilon} + \norm{\tilde{\bm{x}}(0)} \\
            &= -(d^* - d_s -  \norm{\tilde{\bm{x}}(0)} ) + \bar{\varepsilon} - d_\varpi \\
            \overset{\eqref{eq:d_star_condition}}&{\leq} - (\eta-d_s) - (  d_\varpi - \bar{\varepsilon} ) < 0. \numberthis \label{eq:delta_neg_inequality_case_1}
        \end{align*}
        Now consider another scenario where $2d^*-d_s > d(0)$, which means $d_{max}=2d^*-d_s + d_{\varpi}$. Then, \eqref{eq:delta_neg_inequality_general_case} becomes,
        \begin{align*}
            -\delta(t')&+\varrho(t') \\
            &= d^* - (2d^*-d_s+d_{\varpi}) + \bar{\varepsilon} + \norm{\tilde{\bm{x}}(0)} \\
            \overset{\eqref{eq:d_star_condition}}&{\leq} - (\eta-d_s) - ( d_\varpi - \bar{\varepsilon} ) < 0. \numberthis \label{eq:delta_neg_inequality_case_2}
        \end{align*}

        In both scenarios \eqref{eq:delta_neg_inequality_case_1} and \eqref{eq:delta_neg_inequality_case_2}, we have shown that $\sign \big( -\delta(t) + \varrho(t) \big) = -1$, and therefore, it follows that $\dot{d}(t')<0$, which contradicts the implication that $\dot{d}(t') > 0$.

        By deriving a contradiction in each case, we conclude that the assumption that the first exist time $\mathcal{T}_1$ is finite must be false. Therefore, $\mathcal{T}_1=\infty$ and $d(t)\in\mathcal{D}$ for all $t\geq 0$.
    \end{proof}

    \begin{lemma}
        \label{lemma:results_generalization}
        \textbf{Lemma~\ref{lemma:IE_by_IVT}} holds for all $t\geq 0$, and \textbf{Lemma~\ref{lemma:P(t)_is_invertible_IVT}} and \textbf{Proposition~\ref{proposition:xi=x_tilde_IVT}} are valid for all $t> 0$.
    \end{lemma}

    \begin{proof}
        From \textbf{Lemma~\ref{lemma:d_is_bounded_for_all_time}}, we know that the agent-to-target distance $d(t)$ is bounded by \eqref{eq:d_is_bounded_by_IVT} for all $t\geq 0$. Consequently, the inequality \eqref{eq:dgamma/dt>0} for $\mathrm{d}\gamma(t)/\mathrm{d}t$ holds for all $t\geq 0$. It follows that \textbf{Lemma~\ref{lemma:IE_by_IVT}} holds for all $t_0 \geq 0$, and the signal $\bar{\bm{\varphi}}(t)$ is thus PE.

        A direct result of the generalized \textbf{Lemma~\ref{lemma:IE_by_IVT}} is that the Kreisselmeier's regressor $P(t)$ remains non-singular for all $t>0$. Therefore, \textbf{Lemma~\ref{lemma:P(t)_is_invertible_IVT}} holds for all $t > 0$. Consequently, \textbf{Proposition~\ref{proposition:xi=x_tilde_IVT}}, which relies on the non-singularity of $P(t)$, is valid for all $t>0$.
    \end{proof}

    \begin{figure}[t]
        \centering
        \definecolor{c006600}{RGB}{0,102,0}
        \definecolor{cff8000}{RGB}{255,128,0}
        \definecolor{c810080}{RGB}{129,0,128}
        \definecolor{c009999}{RGB}{0,153,153}
        \definecolor{c7f00ff}{RGB}{127,0,255}
        \scalebox{0.465}{
        \begin{tikzpicture}[y=1cm, x=1cm, yscale=1,xscale=1, every node/.append style={scale=1}, inner sep=0pt, outer sep=0pt]
            \path[fill=white] ;
            \path[draw=black] (5.8, 5.8) ellipse (0.8cm and 0.8cm);
            \path[draw=black] (5.8, 5.8) ellipse (2.4cm and 2.4cm);
            \path[draw=black] (5.8, 5.8) ellipse (3.8cm and 3.8cm);
            \path[draw=black] (5.8, 5.8) ellipse (5.8cm and 5.8cm);
            %
            %
            \path[draw=black,line width=0.1cm,miter limit=10.0] (6.0, 5.8) -- (17.2, 5.8);
            \path[draw=black,fill=black,line width=0.1cm,miter limit=10.0] (17.4, 5.8) -- (17.1, 5.7) -- (17.2, 5.8) -- (17.1, 5.9) -- cycle;
            %
            %
            \path[draw=black,fill=black,line width=0.1cm] (5.8, 5.8) ellipse (0.2cm and 0.2cm);
            %
            %
            \begin{scope}[shift={(-0.0, 0.0)}]
                \node[text=black,anchor=south] at (17.9, 5.6){\LARGE $\mathbb{R}$};
            \end{scope}
            %
            %
            \path[draw=black,line width=0.0cm,miter limit=10.0] (6.6, 10.6) -- (6.6, 4.8);
            %
            %
            \begin{scope}[shift={(-0.0, 0.0)}]
                \node[text=black,anchor=south] at (6.6, 4.2){\Large $d_{min}{\color{red}-d^*}$};
            \end{scope}
            %
            %
            \path[draw=black,line width=0.0cm,miter limit=10.0] (8.2, 9.0) -- (8.2, 4.0);
            %
            %
            \begin{scope}[shift={(-0.0, 0.0)}]
                \node[text=black,anchor=south] at (8.2, 3.4){\Large $d_s{\color{red}-d^*}$};
            \end{scope}
            %
            %
            \path[draw=black,line width=0.0cm,miter limit=10.0] (9.6, 5.8) -- (9.6, 3.2);
            %
            %
            \begin{scope}[shift={(-0.0, 0.0)}]
                \node[text=red,anchor=south] (text82) at (9.6, 2.6){\Large ${\color{red}0}$};
            \end{scope}
            %
            %
            \path[draw=black,line width=0.0cm,miter limit=10.0] (11.6, 10.6) -- (11.6, 2.6);
            %
            %
            \begin{scope}[shift={(-0.0, 0.0)}]
                \node[text=black,anchor=south] at (13.5, 0.4){ 
                    \begin{minipage}{\linewidth}
                        \Large \begin{align*}
                            &d_{max}\color{red}
                            -d^* \color{black} \\
                            &=\max(2d^*-d_s,d(0))+d_{\varpi}\color{red}-d^*\color{black}\\
                            &=\max(\color{red}d^*\color{black}-d_s,d(0){\color{red}-d^*})+d_{\varpi}
                        \end{align*}
                    \end{minipage}
                };
            \end{scope}
            %
            %
            \draw [{Triangle[length=3mm,width=2mm]}-{Triangle[length=3mm,width=2mm]}, line width=0.1cm, color=blue] (8.2, 8.2)  -- (11.6, 8.2);
            %
            %
            \begin{scope}[shift={(-0.0, 0.0)}]
                \node[text=blue,anchor=south] at (10.0, 8.3){\Large $\delta(0)$};
            \end{scope}
            %
            %
            \draw [{Triangle[length=3mm,width=2mm]}-{Triangle[length=3mm,width=2mm]}, line width=0.1cm, color=c006600] (6.6, 9.8)  -- (11.6, 9.8);
            %
            %
            \begin{scope}[shift={(-0.0, 0.0)}]
                \node[text=c006600,anchor=south] at (8.9, 9.9){\Large $\delta(t)$};
            \end{scope}
            %
            %
            \path[draw=cff8000,line width=0.1cm,miter limit=10.0] (6.9, 6.1) -- (6.9, 5.5);
            \begin{scope}[shift={(-0.0, 0.0)}]
                \node[text=black,anchor=south] (text218) at (5.2, 6.9) {\Large $d_{min}{\color{red}-d^*}{\color{cff8000}+\underline{\varepsilon}}$};
            \end{scope}
            \path[draw=cff8000,line width=0.0cm,miter limit=10.0] (3.8, 6.8) -- (6.5, 6.8) -- (6.85, 6.1);
            %
            %
            \path[draw=c810080,line width=0.1cm,miter limit=10.0] (11.4, 6.1) -- (11.4, 5.5);
            \begin{scope}[shift={(-0.0, 0.0)}]
                \node[text=black,anchor=south] at (13.8, 4.9){\Large $d_{max}{\color{red}-d^*}{\color{c810080}-\bar{\varepsilon}}$};
            \end{scope}
            \path[draw=c810080,miter limit=10.0] (11.4, 5.55) -- (12.4, 4.8) -- (15.2, 4.8);
            %
            %
            \path[draw=black,line width=0.0cm,miter limit=10.0] (10.3, 7.6) -- (10.3, 5.6);
            \path[draw=magenta,line width=0.1cm,miter limit=10.0] (10.3, 6.1) -- (10.3, 5.5);
            \begin{scope}[shift={(-0.0, 0.0)}]
                \node[text=black,anchor=south] (text190) at (13.8, 4.1){\Large $d_{max}{\color{red}-d^*}{\color{magenta}-d_{\varpi}}$};
            \end{scope}
            \path[draw=magenta,miter limit=10.0] (10.3, 5.55) -- (12.4, 4.0) -- (15.2, 4.0);
            %
            %
            \draw [{Triangle[length=3mm,width=2mm]}-{Triangle[length=3mm,width=2mm]}, line width=0.1cm, color=c009999] (10.3, 7.2)  -- (11.6, 7.2);
            \path[draw=c009999,line width=0.02cm,miter limit=10.0] (11.3, 7.2) -- (14, 7.2);
            %
            %
            \begin{scope}[shift={(-0.0, 0.0)}]
                \node[text=c009999,anchor=south] (text258) at (12.8, 7.3){\Large $\dot{\delta}(t)<0$};
            \end{scope}
            %
            %
            \draw [{Triangle[length=3mm,width=2mm]}-{Triangle[length=3mm,width=2mm]}, line width=0.1cm, color=c7f00ff] (6.6, 8.5)  -- (8.2, 8.5);
            \path[draw=c7f00ff,line width=0.02cm,miter limit=10.0] (4.5, 8.5) -- (6.8, 8.5);
            %
            %
            \begin{scope}[shift={(-0.0, 0.0)}]
                \node[text=c7f00ff,anchor=south] (text278) at (5.5, 8.6){\Large $\dot{\delta}(t)> 0$};
            \end{scope}
        
        \end{tikzpicture}
        }
        \caption{\label{fig:delta_bound}Possible values of $\delta(t)$.}
    \end{figure}

    \begin{figure*}
        \centering
        \includegraphics[width=0.9\textwidth,trim={0mm 232mm 9.526mm 0mm}, clip]{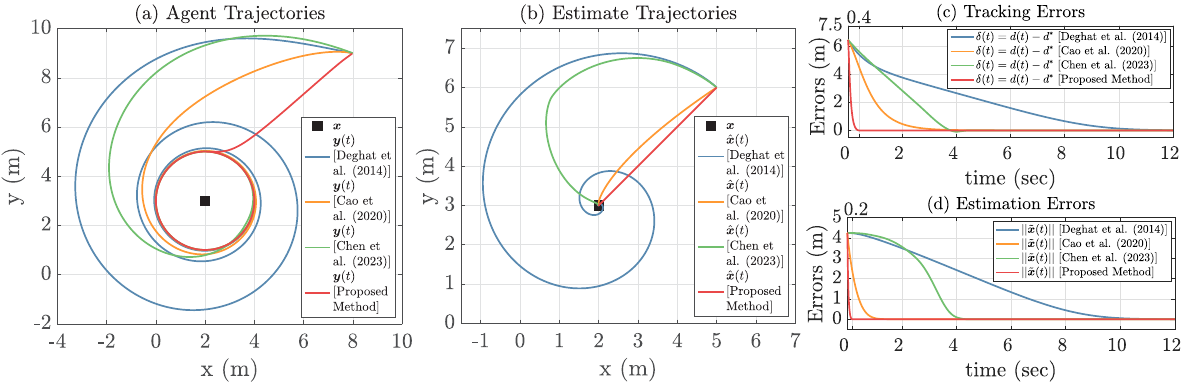}
        \caption{\label{fig:sim_1_ST}Comparative analysis of simulation results benchmarking the performance of the proposed method against established approaches from the literature.}
    \end{figure*}

    \subsection{Main Results}
    \label{sec:Main_Results}

    We now present the main results of this paper.

    \begin{theorem}
        \label{theorem:target_estimation_error_goes_to_zero}
        Suppose \textbf{Assumptions~\ref{assumption:d(0)_initial_condition}-\ref{assumption:estimation_error}} hold. Then, under the target estimator \eqref{eq:def_target_estimator} and the circumnavigation controller \eqref{eq:def_circum_controller}, the norm of the target estimation error $\norm{\tilde{\bm{x}}(t)}$   converge to zero within the strongly predefined time $T_{c,1}$.
    \end{theorem}

    \begin{proof}
        This proof is an application of \cite[Lemma~2.1]{sanchez-torres2018Class}. Let $m\geq 1$ and $p\in (0,1/m]$ be auxiliary variables satisfying $mp=\alpha_1$. Consider the Lyapunov function $V_1 = \norm{\tilde{\bm{x}}(t)}^m$, whose time derivative along the system trajectories is obtained as,
        \begin{align*}
            \scalemath{0.95}{\dot{V}_1} \overset{\eqref{eq:def_target_estimator},\eqref{eq:xi=x_tilde_IVT}}&{=} \scalemath{0.95}{m \norm{\tilde{\bm{x}}(t)}^{m-2} \tilde{\bm{x}}^\top(t) \left( -\frac{\exp(\norm{\tilde{\bm{x}}(t)}^{\alpha_1})}{\alpha_1 T_{c,1}} \bm{\psi}^{\alpha_1}\big(\tilde{\bm{x}}(t)\big) \right)} \\[2pt]
            &= \scalemath{0.95}{-\frac{m}{mpT_{c,1}} \norm{\tilde{\bm{x}}(t)}^{m(1-p)} \exp \big(\norm{\tilde{\bm{x}}(t)}^{mp}\big)} \\[2pt]
            &= \scalemath{0.95}{-\frac{1}{pT_{c,1}} V_{1}^{1-p} \exp \big(V_1^{p}\big) ,} \numberthis \label{eq:V_1_dot}
        \end{align*}
        which takes on the exact form of \eqref{eq:PDT_required_V_dot} in \textbf{Lemma~\ref{lemma:PDT_stability_theorem}}. Therefore, according to \textbf{Lemma~\ref{lemma:PDT_stability_theorem}}, the set $M_1 = \{\tilde{\bm{x}}=\bm{0}_2\}$ is GSPTA with the strongly predefined time $T_{c,1}$. 
    \end{proof}

    \begin{theorem}
        \label{theorem:circum_error_goes_to_zero}
        Suppose \textbf{Assumptions~\ref{assumption:d(0)_initial_condition}-\ref{assumption:estimation_error}} hold. Then, under the target estimator \eqref{eq:def_target_estimator} and the circumnavigation controller \eqref{eq:def_circum_controller}, the tracking error $\delta(t)$ converges to the origin within the strongly predefined time $T_{c,2}$.
    \end{theorem}

    \begin{proof}
        From \textbf{Lemma~\ref{lemma:d_is_bounded_for_all_time}}, we know that $d(t)\in \mathcal{D}$ for all $t\geq 0$, which implies that for $t\leq T_{c,1}$, the signal $\delta(t)=d(t)-d^*$ is bounded and will not escape convergence. Therefore, we may consider the dynamics of the tracking error $\delta(t)$ for the phase after $t\geq T_{c,1}$, which is found as
        \begin{align*}
            \dot{\delta}(t) \overset{\eqref{eq:delta_dynamics}}&{=} - \frac{\exp\big(|\delta(t)-\varrho(t)|^{\alpha_2}\big)}{\alpha_2 T_{c,2}} \sig^{1-\alpha_2}\big(\delta(t) - \varrho(t)\big) \\
            \overset{\textit{Thm.~\ref{theorem:target_estimation_error_goes_to_zero}}}&{=} - \frac{\exp\big(|\delta(t)|^{\alpha_2}\big)}{\alpha_2 T_{c,2}} \sig^{1-\alpha_2}\big(\delta(t)\big), \quad t \geq T_{c,1}. \numberthis \label{eq:delta_dynamics_2}
        \end{align*}
        Introduce auxiliary variables $n\geq 1$ and $q\in (0,1/n]$  such that $nq=\alpha_2$. Then, consider the Lyapunov candidate $V_2=|\delta(t)|^n$, whose time derivative along the system trajectories is obtained as,
        \begin{align*}
            \dot{V}_2 \overset{\eqref{eq:delta_dynamics_2}}&{=} \scalemath{0.95}{n |\delta(t)|^{n-2} \delta(t) \left(-\frac{\exp\big( |\delta(t)|^{\alpha_2} \big)}{\alpha_2 T_{c,2}}\sig^{1-\alpha_2} \big(\delta(t)\big)\right)} \\
            &= \scalemath{0.95}{-\frac{n}{\alpha_2 T_{c,2}} |\delta(t)|^{n-2} \exp\big( |\delta(t)|^{\alpha_2} \big)|\delta(t)|^{2-\alpha_2}} \\
            &= \scalemath{0.95}{-\frac{n}{nq T_{c,2}} \exp\big( |\delta(t)|^{nq} \big)|\delta(t)|^{n(1-q)} }\\
            &= \scalemath{0.95}{- \frac{1}{q T_{c,2}} \exp \big(V_{2}^{q}\big) V_{2}^{1-q} , }\numberthis \label{eq:V_2_dot}
        \end{align*}
        which takes on the exact form of \eqref{eq:PDT_required_V_dot} in \textbf{Lemma~\ref{lemma:PDT_stability_theorem}}. Therefore, according to \textbf{Lemma~\ref{lemma:PDT_stability_theorem}}, the set $M_2=\{\delta = 0\}$ is GSPTA with the strongly predefined time $T_{c,2}$. 
    \end{proof}


    \begin{theorem}
        \label{theorem:angular_velocity}
        Suppose \textbf{Assumptions~\ref{assumption:d(0)_initial_condition}-\ref{assumption:estimation_error}} hold. Then, under the target estimator \eqref{eq:def_target_estimator} and the controller \eqref{eq:def_circum_controller}, the angular velocity $\dot{\theta}(t)$ converges to $\omega^*$ after $t\geq T_{c,2}$.
    \end{theorem}

    \begin{proof}
        From \textbf{Theorem~\ref{theorem:circum_error_goes_to_zero}}, we know that $d(t)=d^*$ for all $t\geq T_{c,2}$, and from \textbf{Lemma~\ref{lemma:results_generalization}}, we know that results in \textbf{Lemma~\ref{lemma:IE_by_IVT}} holds for all $t\geq 0$. Therefore, using \eqref{eq:Dzeta} and \eqref{eq:pe_bound_step1} in \textbf{Lemma~\ref{lemma:IE_by_IVT}}, we obtain that $\mathrm{d}\theta(t)/\mathrm{d}t=k_\omega/d(t)=\omega^*d^*/d^*=\omega^*$ for all $t\geq T_{c,2}$.
    \end{proof}
    
\section{Simulations}
\label{sec:Simulations}
    This section presents a simulated example to demonstrate the performance of the proposed estimation and control algorithms when the target remains stationary. To evaluate the proposed method relative to existing approaches, we benchmark it against notable algorithms from the literature: the pioneering work \cite{deghat2014Localization}, a scalar-estimator-based controller by Cao et al. (2020) \cite{cao2020LowCost}, and a recent finite-time circumnavigation control by Chen et al. (2023) \cite{chen2023Finitetime}. These approaches were selected due to their methodological similarities. 

    To enable a thorough comparison, we briefly summarize the control algorithms from each reference. In Deghat et al. \cite[Eqn.~(5), (6)]{deghat2014Localization}, the estimator and controller are: $\dot{\hat{\bm{x}}}(t) = k_{est}  (I- \bm{\varphi}(t)\bm{\varphi}^\top(t))(\bm{y}(t)-\hat{\bm{x}}(t))$ and $ \bm{u}(t) = k_{\alpha} ( \hat{d}(t) -d^* )\bm{\varphi}(t) + k_\beta \bar{\bm{\varphi}}(t)$. In Cao et al. \cite[Eqn.~(1), (2)]{cao2020LowCost}, the control laws are $\dot{\hat{\rho}}(t)=-\bm{\varphi}^\top(t)\dot{\bm{y}}(t) + k_e \bar{\bm{\varphi}}^\top(t)\dot{\hat{\bm{x}}}(t)$, $\hat{\bm{x}}(t)=\bm{y}(t)+\hat{\rho}(t)\bm{\varphi}(t)$, $\bm{u}(t)=\kappa_\alpha (\hat{\rho}(t)-d^*)\bm{\varphi}(t)+\kappa_\beta \bar{\bm{\varphi}}(t)$. From Chen et al. \cite[Eqn.~(9), (10), (12)]{chen2023Finitetime}, the control algorithms adapted for a single agent are $\dot{\hat{\bm{x}}}(t) = -\kappa_{est} P^\top (t) \sig^{\beta_{1}}(P(t)\hat{\bm{x}}(t)-\bm{y}(t))$, where $\dot{P}(t) = -P(t) + \bar{\bm{\varphi}}(t)\bar{\bm{\varphi}}^\top(t),\;  P(0) = \bm{0}_{2\times 2} $, $\dot{\bm{q}}(t) = -\bm{q}(t) + \bar{\bm{\varphi}}(t)\bar{\bm{\varphi}}^\top (t) \bm{y}(t), \; \bm{q}(0)= \bm{0}_{2}$, $\bm{u}(t) = k_d \sig^{\beta_2}( \hat{d}(t) -d^* )\bm{\varphi}(t) + k_\varphi \bar{\bm{\varphi}}(t)$. Control constants $(k_{est}, k_\alpha, k_\beta)$, $(k_e, \kappa_\alpha, \kappa_\beta)$, and $(\kappa_{est}, k_d, k_\varphi, \beta_1, \beta_2)$ are defined in \cite{deghat2014Localization}, \cite{cao2020LowCost}, and \cite{chen2023Finitetime}, respectively.

    We simulate the solutions to the dynamic system \eqref{eq:agent_kinematics}, driven by the proposed algorithms \eqref{eq:def_target_estimator}, \eqref{eq:def_circum_controller}, as well as the controllers from \cite{deghat2014Localization,cao2020LowCost,chen2023Finitetime}, under the same initial conditions and parameters. The target’s position is $\bm{x}(t) \equiv [2m,3m]^\top $ for all $t\geq 0$, and the agent’s initial position is $\bm{y}(0)=[8m,9m]^\top $. The desired circumnavigation radius is $d^* = 2\, m$, with the agent's initial guess of the target’s position set to $\hat{\bm {x}}(0)=[5m,6m]^\top$. In light of the discussions in \cite{deghat2014Localization,cao2020LowCost,chen2023Finitetime}, the control gains are selected as $T_{c,1}=0.2$ seconds, $T_{c,2}=0.4$ seconds, $k_{est}=k_e=\kappa_{est}=5$,  $k_\omega = k_\beta = \kappa_\beta = k_\varphi = 5$, $k_\alpha = \kappa_\alpha = 1.5$, and $\alpha_1 = \alpha_2 = \beta_1 = \beta_2 = 0.5$.

    The trajectories of the agent presented in Fig.~\ref{fig:sim_1_ST}(a) put in evidence that the agent controlled by the proposed algorithms reaches the desired orbit much faster than those using the algorithms from \cite{deghat2014Localization,cao2020LowCost,chen2023Finitetime}. Further, consistent with \textbf{Theorem~\ref{theorem:circum_error_goes_to_zero}}, Fig.~\ref{fig:sim_1_ST}(c) shows that the tracking error under the proposed method converges to the origin within the strong predefined time $T_{c,2}=0.4$ seconds. Fig.~\ref{fig:sim_1_ST}(b) displays the trajectories of the estimated target position $\hat{\bm{x}}(t)$ generated by the agent under different algorithms. Notably, the estimate from the proposed method moves along an almost straight line (which is also the shortest path) toward the true target position, whereas estimates from the other methods follow more indirect and longer paths. The norms of the estimation errors are graphically illustrated in Fig.~\ref{fig:sim_1_ST}(d). As expected from \textbf{Theorem~\ref{theorem:target_estimation_error_goes_to_zero}}, the estimation error under the proposed method converges to the origin within the strong predefined time $T_{c,1}=0.2$ seconds, demonstrating significantly faster convergence than the methods in \cite{deghat2014Localization} and \cite{chen2023Finitetime}, and a notably faster convergence rate than that in \cite{cao2020LowCost}. However, while the estimation algorithm in Cao et al. \cite{cao2020LowCost} achieves comparably fast convergence rate, it requires either bearing rate measurements or explicit differentiation, rendering it vulnerable to high frequency noise. In contrast, the proposed method is more robust against measurement noise as it relies solely on bearing measurements.


    \begin{figure*}
        \centering
        \includegraphics[width=0.425\textwidth,trim={4.639mm 12.973mm 9.867mm 37.453mm}, clip]{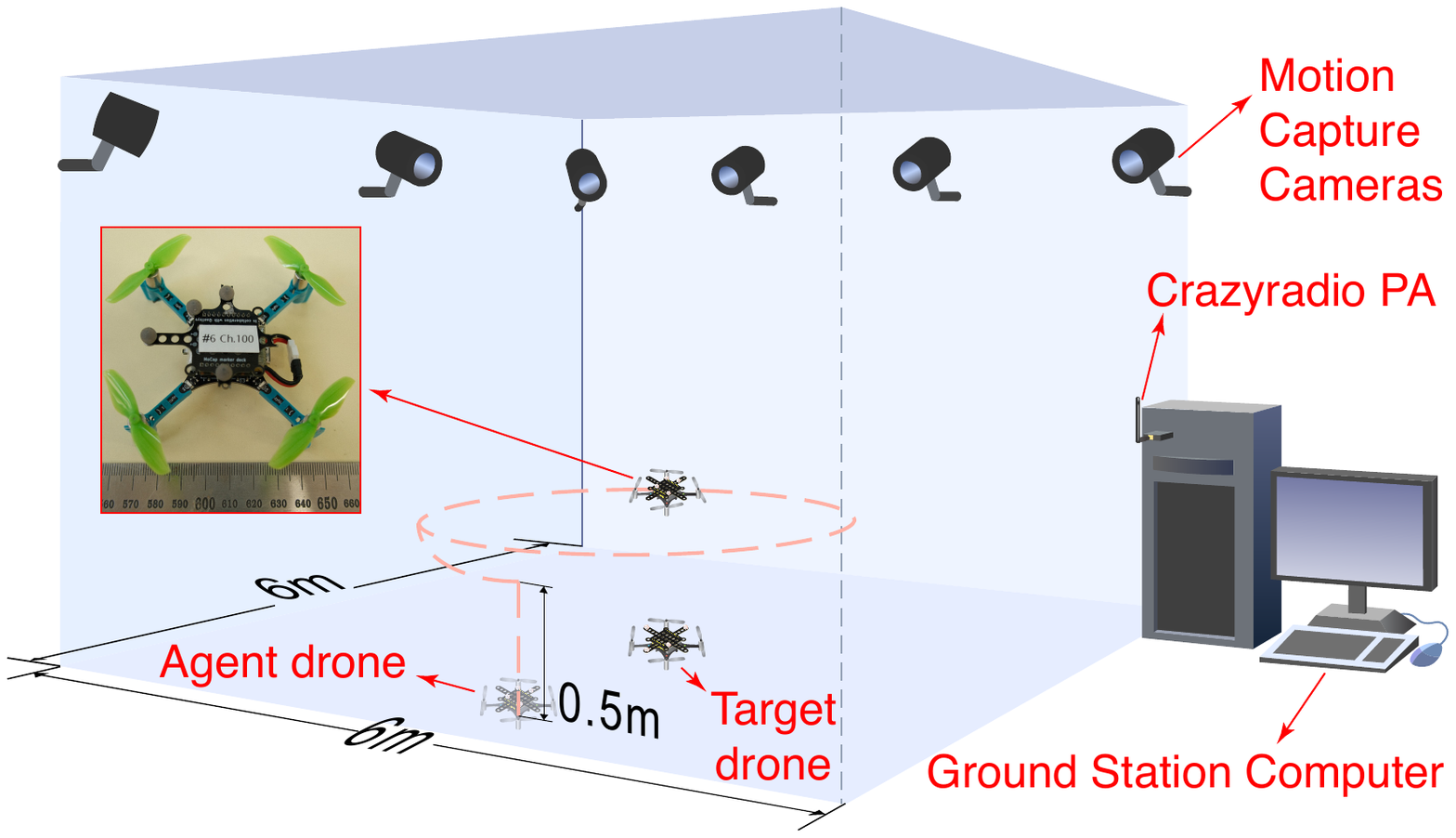}
        \includegraphics[width=0.425\textwidth,trim={0mm 200.281mm 34.655mm 0mm}, clip]{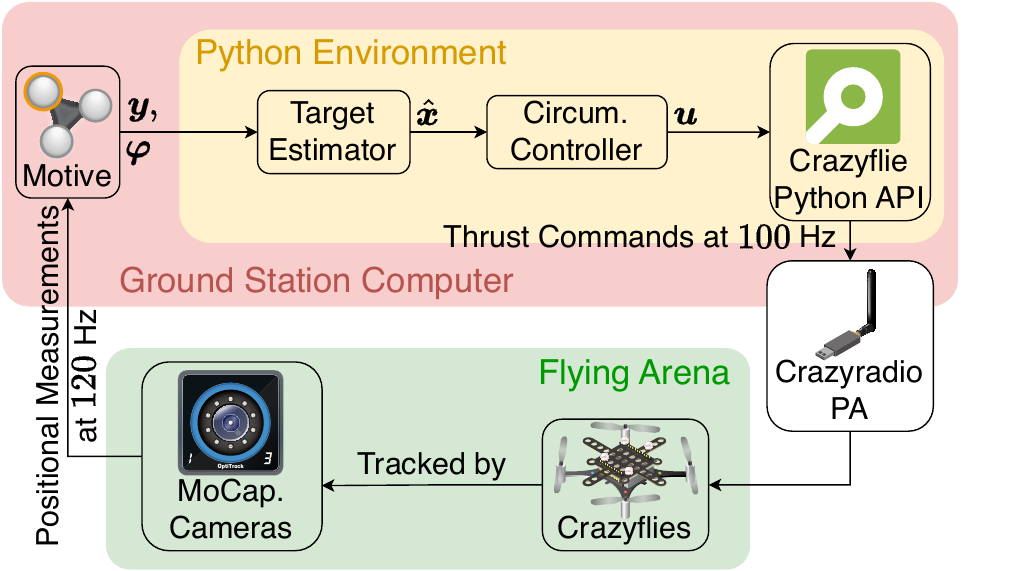}
        \includegraphics[width=0.85\textwidth,trim={0cm 225.1mm 0cm 67.921mm}, clip]{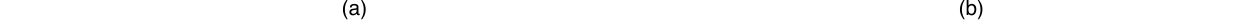}
        \caption{\label{fig:exp_setup}(a) Schematic diagram of the flying arena; (b) Block diagram of the experimental system.}
    \end{figure*}

    \begin{figure*}
        \centering
        \includegraphics[width=0.9\textwidth,trim={0mm 232mm 17.546mm 0mm}, clip]{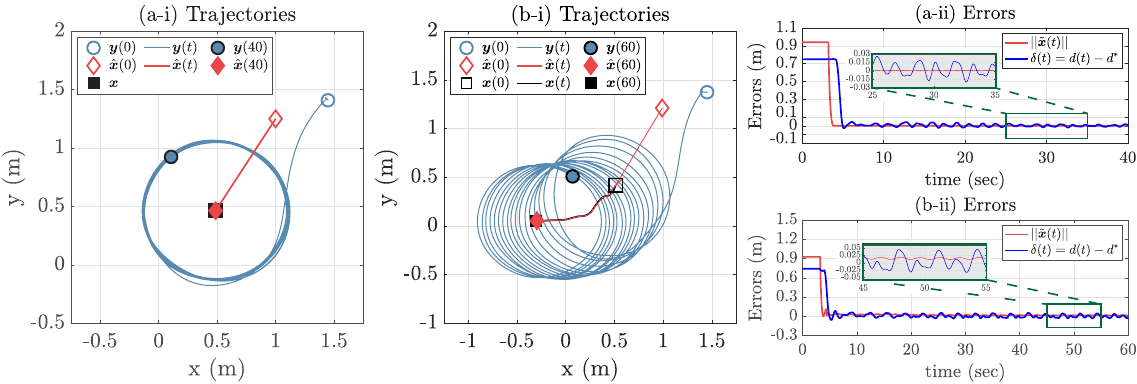}
        \caption{\label{fig:exp_plots}Experimental results demonstrating the effectiveness of the proposed estimation and circumnavigation control algorithms under different scenarios.}
    \end{figure*}
    
\section{Experiments}

    \subsection{Experimental Setup}

    To validate the proposed algorithms in real-world settings, we carried out experiments in a $(6\times 6)\,m^2$ indoor flying arena equipped with $10$ motion capturing cameras (\textit{OptiTrack Prime}$^x$ $13$). Two \textit{Crazyflie 2.1} quadcopters (see Fig.~\ref{fig:exp_setup}(a)) were used: one was assigned as the target, and the other served as the agent that is tasked with the \textit{2-dimensional} BoTLC mission on a constant altitude plane at $0.5m$. Since maintaining a constant altitude for the quadcopter is straightforward, we omit further discussion of the drones' $z$-coordinates. Planar bearing data were extracted from 3D ground-truth positional measurements using the \textit{Motive~2.0} motion capture software. The circumnavigation controller \eqref{eq:def_circum_controller} and the target estimator \eqref{eq:def_target_estimator}, adapted for discrete-time deployment, were computed on a ground station using the 2D bearing data. Control signals were transmitted to the agent via \textit{Crazyradio PA}. The system architecture is illustrated in the block diagram, as presented in Fig.~\ref{fig:exp_setup}(b).

    \subsection{Experiment 1 - Stationary Target}
    \label{sec:Exp_1_ST}

    In the first experiment, we consider a stationary target positioned at $\bm{x}(t) = [0.49m, 0.46m]^\top$ for all $t\geq 0$. The agent drone is initially positioned at $\bm{y}(0)=[1.45m, 1.41m]^\top$. All coordinates are defined in the Earth frame. The control gains are selected as $\alpha_1 = \alpha_2 = 0.5$, $T_{c,1}=1$ second, $T_{c,2}=2$ seconds, $k_\omega = 1$, and the desired circumnavigation radius is set to $d^*=0.6m$. The initial guess of target position is set as $\hat{\bm{x}}(0) = [1m, 1.25m]^\top$. The trajectories and motion trails of the agent drone for a runtime of $t\in[0,40]$ seconds, depicted in Fig.~\ref{fig:exp_plots}(a-i) and Fig.~\ref{fig:echo_photo} respectively, show that the agent drone driven by the proposed algorithms successfully localizes and maintains a circular orbit centered at the target. The performance metrics presented in Fig.~\ref{fig:exp_plots}(a-ii) align with the simulation results and also the expectations from \textbf{Theorems~\ref{theorem:target_estimation_error_goes_to_zero}-\ref{theorem:circum_error_goes_to_zero}}, despite minor oscillations in the error signal $\delta(t)$.

    \subsection{Experiment 2 - Slowly Drifting Target}

    In the second experiment, we consider a moving target initially positioned at $\bm{x}(0)= [0.51m, 0.42m]^\top$ and slowly drifting with a time-varying velocity given by
    \begin{equation*}
        \dot{\bm{x}}(t) = \begin{bmatrix}
            -0.0125 - 0.0125 e^{-0.05t} |\sin (2\pi \cdot 0.03 t)| \, m/s \\
            -0.075 e^{-0.1 t} |\cos (2\pi \cdot 0.03 t)| \, m/s
        \end{bmatrix},
    \end{equation*}
    for a runtime of $t\in [0,60]$ seconds. All other parameters and initial conditions are the same to those in Section~\ref{sec:Exp_1_ST}. The trajectories in Fig.~\ref{fig:exp_plots}(b-i) demonstrate that the agent accurately localizes and circumnavigates the moving target, despite the presence of an unknown time-varying velocity. In Fig.\ref{fig:exp_plots}(b-ii), the convergence of the error signals $\norm{\tilde{\bm{x}}(t)}$ and $\delta(t)$ to small neighborhoods around the origin further confirms the robustness of the proposed control algorithms.

\section{CONCLUDING REMARKS}
\label{sec:conclusion}

    This paper presents a novel target estimation algorithm and an improved circumnavigation controller, both exhibiting strong predefined-time stability, for localizing and orbiting an unknown target based solely on bearing measurements (without relying on explicit differentiation of measurements or any rate measurements). The predefined-time stability of the overall system was rigorously analyzed. Simulation and experimental results showcase satisfactory performance of the proposed control algorithms.

    Future research directions include extending the proposed method to achieve unbiased estimation of moving target(s), accounting for noise, handling intermittent or discrete-time measurements, and applying the approach to multi-agent circumnavigation systems with guaranteed safety.


\addcontentsline{toc}{chapter}{Bibliography}
\bibliographystyle{ieeetr}

\begin{thebibliography}{10}

\bibitem{feldbaum1963Dual}
A.~Feldbaum, ``Dual control theory problems,'' {\em IFAC Proceedings Volumes}, vol.~1, pp.~541--550, June 1963.

\bibitem{rantzer2023Dual}
A.~Rantzer, ``Dual {{Control Revisited}},'' in {\em 2023 62nd {{IEEE Conference}} on {{Decision}} and {{Control}} ({{CDC}})}, pp.~2993--2993, Dec. 2023.

\bibitem{deghat2010Target}
M.~Deghat, I.~Shames, B.~D.~O. Anderson, and C.~Yu, ``Target localization and circumnavigation using bearing measurements in {{2D}},'' in {\em Proc. 49th {{IEEE Conf. Decision Control}} ({{CDC}})}, pp.~334--339, 2010.

\bibitem{deghat2014Localization}
M.~Deghat, I.~Shames, B.~D.~O. Anderson, and C.~Yu, ``Localization and circumnavigation of a slowly moving target using bearing measurements,'' {\em IEEE Trans. Autom. Control}, vol.~59, pp.~2182--2188, 2014.

\bibitem{deghat2015Multitarget}
M.~Deghat, L.~Xia, B.~D.~O. Anderson, and Y.~Hong, ``Multi-target localization and circumnavigation by a single agent using bearing measurements,'' {\em International Journal of Robust and Nonlinear Control}, vol.~25, pp.~2362--2374, Sept. 2015.

\bibitem{chen2023Cooperative}
K.~Chen, G.~Qi, Y.~Li, and A.~Sheng, ``Cooperative localization and circumnavigation of multiple targets with bearing-only measurements,'' {\em Journal of the Franklin Institute}, vol.~360, pp.~9159--9179, Aug. 2023.

\bibitem{sui2024Unbiased}
D.~Sui, M.~Deghat, Z.~Sun, and M.~Greiff, ``Unbiased {{Bearing-Only Localization}} and {{Circumnavigation}} of a {{Constant Velocity Target}},'' {\em IEEE Transactions on Intelligent Vehicles}, pp.~1--15, 2024.

\bibitem{li2018Localization}
R.~Li, Y.~Shi, and Y.~Song, ``Localization and circumnavigation of multiple agents along an unknown target based on bearing-only measurement: {{A}} three dimensional solution,'' {\em Automatica}, vol.~94, pp.~18--25, Aug. 2018.

\bibitem{ma2023Finitetime}
Z.~Ma, Y.~Li, G.~Qi, and A.~Sheng, ``Finite-time distributed localization and multi-orbit circumnavigation with bearing-only measurements,'' {\em Journal of the Franklin Institute}, vol.~360, pp.~6712--6735, July 2023.

\bibitem{chen2022DiscreteTime}
K.~Chen, G.~Qi, Y.~Li, and A.~Sheng, ``Target localization and standoff tracking with discrete-time bearing-only measurements,'' {\em IEEE Transactions on Circuits and Systems II: Express Briefs}, pp.~1--1, 2022.

\bibitem{wang2024Target}
D.~Wang, Z.~Zhang, Y.~Zhao, and C.~Xu, ``Target tracking with circumnavigation scheme using discrete bearing and control input,'' {\em IEEE Trans. Aerosp. Electron. Syst.}, pp.~1--16, 2024.

\bibitem{greiff2022Target}
M.~Greiff, M.~Deghat, Z.~Sun, and A.~Robertsson, ``Target {{Localization}} and {{Circumnavigation With Integral Action}} in {{R}}{$^{2}$},'' {\em IEEE Control Systems Letters}, vol.~6, pp.~1250--1255, 2022.

\bibitem{sui2024Adaptive}
D.~Sui, M.~Deghat, Z.~Sun, and M.~Eskandari, ``Adaptive bearing-only target localization and circumnavigation under unknown wind disturbance: Theory and experiments,'' {\em IEEE Robotics and Automation Letters}, vol.~9, pp.~11321--11328, Dec. 2024.

\bibitem{wang2022Target}
W.~Wang, X.~Chen, J.~Jia, and Z.~Fu, ``Target localization and encirclement control for multi-{{UAVs}} with limited information,'' {\em IET Control Theory \& Applications}, vol.~16, pp.~1396--1404, Sept. 2022.

\bibitem{hu2022BearingOnly}
B.~Hu and H.~Zhang, ``Bearing-only motional target-surrounding control for multiple unmanned surface vessels,'' {\em IEEE Trans. Ind. Electron.}, vol.~69, pp.~3988--3997, Apr. 2022.

\bibitem{deghat2012Target}
M.~Deghat, E.~Davis, T.~See, I.~Shames, B.~D.~O. Anderson, and C.~Yu, ``Target localization and circumnavigation by a non-holonomic robot,'' in {\em Proc. 2012 IEEE/RSJ Int. Conf. Intell. Robots Syst.}, (Vilamoura-Algarve, Portugal), pp.~1227--1232, IEEE, Oct. 2012.

\bibitem{cao2023Bearingonly}
S.~Cao and G.~Duan, ``Bearing-only circumnavigation based on fully actuated system approach,'' in {\em Proc. 2023 2nd Conf. Fully Actuated Syst. Theory Appl. (CFASTA)}, pp.~361--366, IEEE, July 2023.

\bibitem{dou2020Target}
L.~Dou, C.~Song, X.~Wang, L.~Liu, and G.~Feng, ``Target localization and enclosing control for networked mobile agents with bearing measurements,'' {\em Automatica}, vol.~118, p.~109022, Aug. 2020.

\bibitem{chen2022Multicircular}
K.~Chen, G.~Qi, Y.~Li, and A.~Sheng, ``Target localization and multicircular circumnavigation with bearing-only measurements,'' {\em Int. J. Adapt. Control Signal Process.}, p.~acs.3517, Oct. 2022.

\bibitem{sui2024CollisionFree}
D.~Sui and M.~Deghat, ``Collision-free zero-communication finite-time target localization and circumnavigation by multiple agents using bearing-only measurements,'' {\em J. Franklin Inst.}, p.~106972, June 2024.

\bibitem{chen2023Finitetime}
K.~Chen, G.~Qi, Y.~Li, and A.~Sheng, ``Finite-time target localization and multicircular circumnavigation with bearing-only measurements,'' {\em Journal of the Franklin Institute}, vol.~360, pp.~6338--6356, June 2023.

\bibitem{zhou2023FiniteTime}
Z.~Zhou, B.~Chen, and J.~Hu, ``On {{Finite-Time Multi-Target Localization}} and {{Circumnavigation Control Using Bearing Measurements}},'' in {\em Proc. 42nd Chinese Control Conf. (CCC)}, pp.~5676--5681, 2023.

\bibitem{zhao2023Bearing}
D.~Zhao and Q.~Geng, ``Bearing measurements based moving target localization and circumnavigation,,'' in {\em Proc. 2023 42nd Chinese Control Conf. (CCC)}, pp.~2887--2892, IEEE, July 2023.

\bibitem{chowdhary2013Concurrent}
G.~Chowdhary, T.~Yucelen, M.~M{\"u}hlegg, and E.~N. Johnson, ``Concurrent learning adaptive control of linear systems with exponentially convergent bounds,'' {\em Int. J. Adapt. Control Signal Process.}, vol.~27, pp.~280--301, Apr. 2013.

\bibitem{pan2018Composite}
Y.~Pan and H.~Yu, ``Composite learning robot control with guaranteed parameter convergence,'' {\em Automatica}, vol.~89, pp.~398--406, Mar. 2018.

\bibitem{aranovskiy2023PreservingExcitation}
S.~Aranovskiy, R.~Ushirobira, M.~Korotina, and A.~Vedyakov, ``On {{Preserving-Excitation Properties}} of {{Kreisselmeier}}'s {{Regressor Extension Scheme}},'' {\em IEEE Transactions on Automatic Control}, vol.~68, pp.~1296--1302, Feb. 2023.

\bibitem{sanchez-torres2018Class}
J.~D. {S{\'a}nchez-Torres}, D.~{G{\'o}mez-Guti{\'e}rrez}, E.~L{\'o}pez, and A.~G. Loukianov, ``A class of predefined-time stable dynamical systems,'' {\em IMA J. Math. Control Inf.}, vol.~35, pp.~i1--i29, Mar. 2018.

\bibitem{filippov1988Differential}
A.~F. Filippov, {\em Differential {{Equations}} with {{Discontinuous Righthand Sides}}}, vol.~18 of {\em Mathematics and {{Its Applications}}}.
\newblock Dordrecht: Springer Netherlands, 1988.

\bibitem{sastry1994Adaptive}
S.~Sastry, M.~Bodson, and J.~F. Bartram, {\em Adaptive Control: Stability, Convergence, and Robustness}.
\newblock Prentice-Hall, 1994.

\bibitem{cao2020LowCost}
S.~Cao, R.~Li, Y.~Shi, and Y.~Song, ``A low-cost estimator for target localization and circumnavigation using bearing measurements,'' {\em Journal of the Franklin Institute}, vol.~357, pp.~9654--9672, Sept. 2020.

\end{thebibliography}

\end{document}